\documentclass[12pt, reqno, fleqn]{amsart}
\usepackage{amsmath, amsthm, amscd, amsfonts, amssymb, graphicx, color}
\usepackage[bookmarksnumbered, colorlinks, plainpages]{hyperref}

\input{mathrsfs.sty}

\newtheorem{theorem}{Theorem}[section]

\newtheorem{corollary}[theorem]{Corollary}
\theoremstyle{definition}

\newtheorem{example}[theorem]{Example}

\theoremstyle{remark}
\newtheorem{remark}[theorem]{Remark}
\numberwithin{equation}{section}

\begin{document}

\title{Noncommutative Chebyshev  inequality involving the Hadamard product}

\author[ M. Bakherad, S.S. Dragomir]{ Mojtaba Bakherad$^1$ and Silvestru Sever Dragomir$^2$}

\address{$^1$Department of Mathematics, Faculty of Mathematics, University of Sistan and Baluchestan, P.O. Box 98135-674, Zahedan, Iran.}
\email{mojtaba.bakherad@yahoo.com; bakherad@member.ams.org}
\address{$^1$School of Computer Science and Mathematics, Victoria
University of Technology, PO Box 14428, Melbourne City
MC, Victoria 8001, Australia.}
\email{sever.dragomir@vu.edu.au}



\begin{abstract}
We present several operator extensions of the Chebyshev
inequality for Hilbert space operators. The main version deals with the synchronous Hadamard property for
Hilbert space operators. Among other inequalities, it is shown that if
 ${\mathscr A}$ is a $C^*$-algebra, $T$ is a compact Hausdorff
space equipped with a Radon measure $\mu$ as a totaly order set, then
{\footnotesize\begin{align*}
\int_{T} \alpha(s) d\mu(s)&\int_{T}\alpha(t)(A_t\circ B_t) d\mu(t)\\&\geq\Big{(}\int_{T}\alpha(t) (A_tm_{r,\alpha} B_t) d\mu(t)\Big{)}\circ\Big{(}\int_{T}\alpha(s) (A_sm_{r,1-\alpha} B_s) d\mu(s)\Big{)},
\end{align*}}
where  $\alpha\in[0,1]$, $r\in[-1,1]$ and $(A_t)_{t\in T}, (B_t)_{t\in T} $ are positive increasing fields in
$\mathcal{C}(T,\mathscr A)$.
\end{abstract} \maketitle
\section{Introduction and preliminaries}

Let ${\mathbb B}({\mathscr H})$ denote the $C^*$-algebra of all
bounded linear operators on a complex Hilbert space ${\mathscr H}$.  In the case when ${\rm dim}{\mathscr H}=n$, we identify ${\mathbb
B}({\mathscr H})$ with the matrix algebra $\mathbb{M}_n$ of all
$n\times n$ matrices with entries in the complex field. An operator $A\in{\mathbb B}({\mathscr H})$ is called positive $(A\geq0)$
 if $\langle Ax,x\rangle\geq0$
for all $x\in{\mathscr H }$.  The set of all positive  operators  is denoted by ${\mathbb B}({\mathscr H})_+$. For
selfadjoint operators $A, B\in{\mathbb B}({\mathscr H})$, we say
$B\geq A$  if $B-A\geq0$.\\   The Gelfand map $f(t)\mapsto f(A)$ is an
isometrically $*$-isomorphism between the $C^*$-algebra
$C({\rm sp}(A))$ of continuous functions on the spectrum ${\rm sp}(A)$
of a selfadjoint operator $A$ and the $C^*$-algebra generated by $A$ and the identity operator $I$. If $f, g\in C({\rm sp}(A))$, then
$f(t)\geq g(t)\,\,(t\in{\rm sp}(A))$ implies that $f(A)\geq g(A)$.\\
If $f$ be a continuous real valued function on an interval $J$. The
function $f$ is called operator monotone  if
$A\leq B$ implies $f(A)\leq f(B)$  for all
$A, B\in {\mathbb B}({\mathscr H})$ with spectra in $J$.
Given an orthonormal basis $\{e_j\}$ of a Hilbert space ${\mathscr H}$, the Hadamard product $A \circ B$ of two operators $A, B \in {\mathbb B}({\mathscr H})$ is defined by $\langle A\circ B e_i, e_j\rangle = \langle Ae_i, e_j\rangle \langle Be_i, e_j\rangle$. It is known that the Hadamard
product can be presented by filtering the
tensor product $A \otimes B$ through a positive linear map. In fact,
$
 A\circ B=U^*(A\otimes B)U,
$
where $U:{\mathscr H}\to {\mathscr H}\otimes{\mathscr H}$ is the isometry
defined by $Ue_j=e_j\otimes e_j$; see \cite{FUJ}. For matrices, one easily observe \cite{STY} that the Hadamard product of $A=(a_{ij})$ and
$B=(b_{ij})$ is $A\circ B=(a_{ij}b_{ij})$, a principal
submatrix of the tensor product $A\otimes B=(a_{ij}B)_{1 \leq i,j\leq n}$. From now on when we deal with the Hadamard product of operators, we explicitly assume that we fix an orthonormal basis.

The axiomatic theory of operator means has been developed by Kubo and Ando \cite{ando}. An operator mean is a binary operation $\sigma$ defined on the set of strictly positive operators, if the following conditions hold:
\begin{itemize}
\item [(1)] $A\leq C, B\leq D$ imply
 $A\sigma B\leq C\sigma D$;
\item [(2)] $A_n\downarrow A, B_n\downarrow B$ \hspace{.2cm}imply
 $A_n\sigma B_n\downarrow A\sigma B$, where $A_n\downarrow A$ means that $A_1\geq A_2\geq \cdots$ and $A_n\rightarrow A$ as $n\rightarrow\infty$ in the strong operator topology;
\item [(3)] $T^*(A\sigma B)T\leq (T^*AT)\sigma (T^*BT)\hspace{.5cm}(T\in{\mathbb B}({\mathscr H}));$
\item [(4)] $I\sigma I=I.$
\end{itemize}
 There exists an affine order isomorphism between the class of
operator means and the class of positive operator monotone functions
$f$ defined on $(0,\infty)$ with $f(1)=1$ via
$f(t)I=I\sigma(tI)\hspace{.1cm}(t>0)$. In addition, $A\sigma
B=A^{1\over 2}f(A^{-1\over2}BA^{-1\over2})A^{1\over2}$ for all
strictly positive operators $A, B$. The operator monotone function
$f$ is called the representing function of $\sigma$. Using a limit
argument by $A_\varepsilon=A+\varepsilon I$, one can extend the
definition of $A\sigma B$ to positive operators.  An operator mean $\sigma$ is symmetric if $A\sigma B=B\sigma A$ for all
 $A, B\in{\mathbb B}({\mathscr H})_+$. For a symmetric operator mean $\sigma$, a parametrized operator mean $\sigma_t,\,\, 0 \leq t\leq 1$ is called an interpolational path for $\sigma$ if it satisfies
\begin{itemize}
\item [(1)] $A\sigma_0 B=A$,
$A\sigma_{1/2}B=A\sigma B$, and $A\sigma_1B=B$;
\item [(2)] $(A\sigma_p B)\sigma(A\sigma_qB)=A\sigma_{\frac{p+q}{2}}B$ for all  $p,q\in[0,1]$;
\item [(3)] The map $t\in[0,1]\mapsto A\sigma_tB$ is norm continuous for each $A$ and $B$.
\end{itemize}
It is easy to see that the set of all $r\in[0, 1]$ satisfying
\begin{align}\label{fujii23}
(A\sigma_p B)\sigma_r(A\sigma_qB)=A\sigma_{rp+(1-r)q}B
 \end{align}
for all $p, q$ is a convex subset of $[0, 1]$ including $0$ and $1$.
The power means interpolational paths are
 {\footnotesize\begin{align*}
Am_{r,t}B=A^\frac{1}{2}\left({1-t+t(A^\frac{-1}{2}BA^\frac{-1}{2}})^r\right)^\frac{1}{r}A^\frac{1}{2}\qquad(t\in[0,1]).
 \end{align*}}
 In particular, we have the operator weighted arithmetic mean $Am_{1,t}B=A\nabla_t B=(1-t)A+tB,$ the operator weighted geometric mean $Am_{0,t}B=A\sharp_t B$ and the operator weighted harmonic  mean $Am_{-1,t}B=A!_t B=\left((1-t)A^{-1}+tB^{-1}\right)^{-1}$. The representing function $F_{r,t}$ for $m_{r,t}$ is defined as $
F_{r,t}(x)=1m_{r,t}x=(1-t+tx^r)^\frac{1}{r}\qquad(x>0).
 $
Let us consider the real sequences $a=(a_1, \cdots, a_n)$, $b=(b_1,
\cdots, b_n)$ and the non-negative sequence $w=(w_1, \cdots, w_n)$.
Then the weighed Chebyshev function is defined by
 ${\footnotesize
 T(w; a, b):=\sum_{i=1}^nw_{i}\sum_{i=1}^nw_ja_jb_j
 -\sum_{i=1}^nw_ia_i\sum_{j=1}^nw_jb_j.
}$
In 1882, Chebyshev  \cite{che1} proved that if $a$ and $b$ are
monotone in the same sense, then
\begin{align}\label{12}
 \sum_{i=1}^n\omega_{i}\sum_{j=1}^n\omega_ja_jb_j
 \geq\sum_{i=1}^n\omega_ia_i\sum_{j=1}^n\omega_jb_j.
 \end{align}
 Behdzed in \cite{behzad} extended inequality \eqref{12} to
{\footnotesize\begin{align}
 \sum_{i=1}^n\omega_{i}\sum_{j=1}^n\nu_ja_jb_j+\sum_{i=1}^n\nu_{i}\sum_{j=1}^n\omega_ja_jb_j\geq\sum_{i=1}^n\nu_ia_i\sum_{j=1}^n\omega_jb_j+
 \sum_{i=1}^n\omega_ia_i\sum_{j=1}^n\nu_jb_j,
 \end{align}}
 where $a_1\leq \cdots\leq a_n$, $b_1\leq
\cdots\leq b_n$  and   $\omega_1, \cdots, \omega_n$, $\nu_1, \cdots, \nu_n$ are nonnegative real numbers.
Some integral generalizations of the Chebyshev inequality was given by Barza,
Persson and Soria \cite{BPS}. The Chebyshev inequality is a complement of the Gr\" uss inequality; see \cite{MR} and references therein. Dragomir
presented some Chebyshev inequalities for selfadjoint operators
acting on Hilbert spaces in \cite{abd,DRA2}.\\ A related notion is synchronicity.
Recall that two continuous functions $f, g:J\rightarrow \mathbb{R}$
are synchronous on an interval $J$, if
 \begin{align*}
 \big{(}f(t)-f(s)\big{)}\big{(}g(t)-g(s)\big{)}\geq 0
 \end{align*}
for all $s,t \in J$. It is obvious that, if $f,g$ are monotonic and
have the same monotonicity, then they are synchronic.

Let ${\mathscr A}$ be a $C^*$-algebra of operators acting on  a Hilbert space, let $T$ be a compact
Hausdorff space and  $\mu(t)$ be a Radon
measure on $T$. A field $(A_t)_{t\in T}$ of operators in ${\mathscr
A}$ is called a continuous field of operators if the function $t
\mapsto A_t$ is norm continuous on $T$ and the function $t \mapsto
\|A_t\|$ is integrable, one can form the Bochner integral
$\int_{T}A_t{\rm d}\mu(t)$, which is the unique element in
${\mathscr A}$ such that
$\varphi\left(\int_TA_t{\rm d}\mu(t)\right)=\int_T\varphi(A_t){\rm d}\mu(t)$
for every linear functional $\varphi$ in the norm dual ${\mathscr
A}^*$ of ${\mathscr A}$.
 By \cite{bakh} for operators $B,A_t\in{\mathscr A}$ we have
{\footnotesize\begin{align}\label{suit}
\int_{T} (A_t\circ B)d\mu(t)
=\int_{T} A_td\mu(t)\circ B\qquad (A_t, B\in{\mathscr A}).
\end{align}}
We say that two fields $(A_t)_{t\in T}$ and $(B_t)_{t\in T}$
  have the synchronous Hadamard property if
\begin{align*}
\big{(}A_t-A_s\big{)}\circ\big{(}B_t-B_s\big{)}\geq0
\end{align*}
for all $s, t\in T$.\\
In \cite{bakh}, the authors showed that
{\footnotesize\begin{align}\label{22big}
\int_{T} \alpha(s) d\mu(s)\int_{T}\alpha(t)(A_t\circ B_t) d\mu(t)\geq\Big{(}\int_{T}\alpha(t) A_t d\mu(t)\Big{)}\circ\Big{(}\int_{T}\alpha(s) B_s d\mu(s)\Big{)},
 \end{align}}
where ${\mathscr A}$ is a $C^*$-algebra, $T$ is a compact
 Hausdorff space equipped with a Radon measure $\mu$,  $(A_t)_{t\in T}$ and $(B_t)_{t\in T}$ are fields in $\mathcal{C}(T,\mathscr A)$ with the synchronous Hadamard property and $\alpha: T\rightarrow [0, +\infty)$ is a measurable function. They also presented
 {\footnotesize\begin{align}\label{bigbag}
\int_{T} \alpha(s) d\mu(s)\int_{T}\alpha(t)(A_t\circ B_t) d\mu(t)\geq\Big{(}\int_{T}\alpha(t) (A_t\sharp_\mu B_t) d\mu(t)\Big{)}\circ\Big{(}\int_{T}\alpha(s) (A_s\sharp_{1-\mu} B_s) d\mu(s)\Big{)},
\end{align}}
where ${\mathscr A}$ is a $C^*$-algebra, $T$ is a compact Hausdorff
space equipped with a Radon measure $\mu$ as a totaly order set,
$(A_t)_{t\in T}, (B_t)_{t\in T} $ are positive increasing fields in
$\mathcal{C}(T,\mathscr A)$, $\alpha: T\rightarrow [0,
+\infty)$ is a measurable function  and $\mu\in[0,1]$.

In this paper, we provide several operator extensions of the
Chebyshev inequality of the form \eqref{22big} and \eqref{bigbag}. We present our main
results dealing with the Hadamard product of Hilbert space operators.

~~~~~~~~~~~~~~~~~~~~~~~~~~~~~~~~~~~~~~~~~~~~~~~~~~~~~~~~~~~~~~~~~~~~~~~~~~~~~~~~~~~~~~~~~~~~~~~~~~~~~~~~~~~~~~~~~~~~~~~~~~~~~~~~~~~~~~~~~~~~~~~~~~~~~~~~~~~~~~~

\section{Chebyshev inequality involving Hadamard product}

This section is devoted to presentation of some operator Chebyshev
inequalities dealing with the Hadamard product.
\noindent The first result reads as follows.
\begin{theorem}\label{20}
 Let ${\mathscr A}$ be a $C^*$-algebra, $T$ be a compact
 Hausdorff space equipped with a Radon measure $\mu$, let $(A_t)_{t\in T}$ and $(B_t)_{t\in T}$ be fields in $\mathcal{C}(T,\mathscr A)$ with the synchronous Hadamard property and let $\alpha, \beta: T\rightarrow [0, +\infty)$ be measurable functions. Then
{\footnotesize\begin{align}\label{22}
\int_{T}& \alpha(s) d\mu(s)\int_{T}\beta(t)(A_t\circ B_t) d\mu(t)+ \int_{T} \beta(s) d\mu(s)\int_{T}\alpha(t)(A_t\circ B_t) d\mu(t)\nonumber\\&\geq\Big{(}\int_{T}\alpha(t) A_t d\mu(t)\Big{)}\circ\Big{(}\int_{T}\beta(s) B_s d\mu(s)\Big{)}
+\Big{(}\int_{T}\beta(t) A_t d\mu(t)\Big{)}\circ\Big{(}\int_{T}\alpha(s) B_s d\mu(s)\Big{)}.
 \end{align}}
 \end{theorem}
\begin{proof}
 We put
 {\footnotesize\begin{eqnarray*}
 \Lambda&=&\int_{T} \alpha(s) d\mu(s)\int_{T}\beta(t)(A_t\circ B_t) d\mu(t)+ \int_{T} \beta(s) d\mu(s)\int_{T}\alpha(t)(A_t\circ B_t) d\mu(t)\\&&
 -\Big{(}\int_{T}\alpha(t) A_t d\mu(t)\Big{)}\circ\Big{(}\int_{T}\beta(s) B_s d\mu(s)\Big{)}
-\Big{(}\int_{T}\beta(t) A_t d\mu(t)\Big{)}\circ\Big{(}\int_{T}\alpha(s) B_s d\mu(s)\Big{)}
 \end{eqnarray*}}
 then
 {\footnotesize\begin{eqnarray*}
 \Lambda&=&
 \int_{T}\int_{T} \alpha(s)\beta(t)(A_t\circ B_t) d\mu(t)d\mu(s)+\int_{T}\int_{T} \beta(s)\alpha(t)(A_t\circ B_t) d\mu(t)d\mu(s)\\&&-\int_{T}\Big{(}\int_{T}\alpha(t) A_t d\mu(t)\Big{)}\circ \beta(s)B_s d\mu(s)\int_{T}\Big{(}\int_{T}\beta(t) A_t d\mu(t)\Big{)}\circ \alpha(s)B_s d\mu(s)\\&&
  \qquad\qquad \qquad\qquad \qquad\qquad\qquad\qquad \qquad\qquad\qquad \qquad \qquad (\normalsize\textrm {by~} \eqref{suit})\\&
 =&\int_{T} \int_{T} \alpha(s)\beta(t)(A_t\circ B_t) d\mu(t)d\mu(s)+\int_{T} \int_{T} \beta(s)\alpha(t)(A_t\circ B_t) d\mu(t)d\mu(s)\\&&-\int_{T}\int_{T}\alpha(t)\beta(s)( A_t \circ B_s) d\mu(t)d\mu(s)-\int_{T}\int_{T}\beta(t)\alpha(s)( A_t \circ B_s) d\mu(t)d\mu(s)\\&&
 \qquad\qquad \qquad\qquad \qquad\qquad\qquad\qquad \qquad\qquad\qquad \qquad \qquad (\normalsize\textrm {by~} \eqref{suit})\\&
  =&\int_{T} \int_{T} \Big[\alpha(s)\beta(t)(A_t\circ B_t) + \beta(s)\alpha(t)(A_t\circ B_t) \\&&-\alpha(t)\beta(s)( A_t \circ B_s) -\beta(t)\alpha(s)( A_t \circ B_s)\Big] d\mu(t)d\mu(s)\\&=&
  {1\over2}\Bigg(\int_{T} \int_{T} \Big[\alpha(s)\beta(t)(A_t\circ B_t) + \beta(s)\alpha(t)(A_t\circ B_t) \\&&-\alpha(t)\beta(s)( A_t \circ B_s) -\beta(t)\alpha(s)( A_t \circ B_s)\Big] d\mu(t)d\mu(s)\\&&+
\int_{T} \int_{T} \Big[\alpha(s)\beta(t)(A_t\circ B_t) + \beta(s)\alpha(t)(A_t\circ B_t) \\&&-\alpha(t)\beta(s)( A_t \circ B_s) -\beta(t)\alpha(s)( A_t \circ B_s)\Big] d\mu(t)d\mu(s)\Bigg)
  \end{eqnarray*}}
 Now, if we interchange $s$ with $t$ in the second expression of the last equation, then we get
  {\footnotesize\begin{eqnarray*}
  \Lambda&\geq&{1\over2}\Bigg(\int_{T} \int_{T} \Big[\alpha(s)\beta(t)(A_t\circ B_t) + \beta(s)\alpha(t)(A_t\circ B_t) \\&&-\alpha(t)\beta(s)( A_t \circ B_s) -\beta(t)\alpha(s)( A_t \circ B_s)\Big] d\mu(t)d\mu(s)\\&&+
 \int_{T} \int_{T} \Big[\alpha(t)\beta(s)(A_s\circ B_s) + \beta(t)\alpha(s)(A_s\circ B_s) \\&&-\alpha(s)\beta(t)( A_s \circ B_t) -\beta(s)\alpha(t)( A_s \circ B_t)\Big] d\mu(s)d\mu(t)\Bigg)\\&=&{1\over2}\Bigg(\int_{T} \int_{T} \Big[\alpha(s)\beta(t)(A_t\circ B_t) + \beta(s)\alpha(t)(A_t\circ B_t) \\&&-\alpha(t)\beta(s)( A_t \circ B_s) -\beta(t)\alpha(s)( A_t \circ B_s)\Big] d\mu(t)d\mu(s)\\&&+
 \int_{T} \int_{T} \Big[\alpha(t)\beta(s)(A_s\circ B_s) + \beta(t)\alpha(s)(A_s\circ B_s) \\&&-\alpha(s)\beta(t)( A_s \circ B_t) -\beta(s)\alpha(t)( A_s \circ B_t)\Big] d\mu(t)d\mu(s)\Bigg)\end{eqnarray*}}

 {\footnotesize\begin{eqnarray*}&=&
  {1\over2}\int_{T}\int_{T}\Big{[}\beta(s)\alpha(t)\big{(}A_t-A_s\big{)}\circ\big{(}B_t-B_s\big{)} +
 \alpha(s)\beta(t)\big{(}A_t-A_s\big{)}\circ\big{(}B_t-B_s\big{)}\Big{]} d\mu(t) d\mu(s)
 \\&\geq&0.\qquad(\normalsize\textrm{since the fields}\hspace{.1cm}(A_t)\hspace{.1cm} \normalsize\textrm{and}\hspace{.1cm} (B_t)\hspace{.1cm}
 \normalsize\textrm{have synchronous Hadamard property})
 \end{eqnarray*}}
 \end{proof}
 In the discrete case $T=\{1,\cdots, n\}$, let $\alpha(i)=\omega_i$ and $\beta(i)=\nu_i$, where $\omega_i, \nu_i\geq0$ $(1\leq i\leq n)$.
 Then Theorem \ref{20} forces the following corollary.
\begin{corollary}\label{31}
 Suppose that  $A_j,B_j\in{\mathbb B}({\mathscr H})\,\,(1\leq j\leq n)$ are selfadjoint operators with the synchronous Hadamard property and
 $\omega_1, \cdots, \omega_n$,  $\nu_1, \cdots, \nu_n$ are positive numbers. Then
 {\footnotesize\begin{align}\label{13}
 \sum_{ i=1}^n \omega_i\sum_{ j=1}^n\nu_j (A_j\circ B_j)&+\sum_{ i=1}^n \nu_i\sum_{ j=1}^n\omega_j (A_j\circ B_j)\nonumber\\&\geq\Big{(}\sum_{ i=1}^n \omega_i A_i\Big{)}\circ\Big{(}\sum_{ j=1}^n
 \nu_j B_j\Big{)}+\Big{(}\sum_{ i=1}^n \nu_i A_i\Big{)}\circ\Big{(}\sum_{ j=1}^n
 \omega_j B_j\Big{)}.
\end{align}}
\end{corollary}
\begin{example}
If $f_1, f_2,g_1, g_2\in \mathbf{L}^1(\mathbb{R})$ such that $f_1, f_2$ are  increasing   and $g_1, g_2$ are decreasing  on $\mathbb{R}$, then we put
$A_t=\left(\begin{array}{cc}
 f_1(t)&h(t)\\
 0&g_1(t)
 \end{array}\right)$ and
 $B_t=\left(\begin{array}{cc}
 f_2(t)&0\\
 k(t)&g_2(t)
 \end{array}\right)$
 in which $h, k\in\mathbf{L}^1(\mathbb{R})$ are arbitrary  and $t\in \mathbb{R}$. It follows from  $\big(f_1(t)-f_1(s)\big)\big(f_2(t)-f_2(s)\big)$ and
 $\big(g_1(t)-g_1(s)\big)\big(g_2(t)-g_2(s)\big)$ are positive for all $s,t\in \mathbb{R}$ that the matrix
{\footnotesize \begin{align*}
\big{(}A_t-A_s\big{)}\circ\big{(}B_t-B_s\big{)}=\left(\begin{array}{cc}
 \left(f_1(t)-f_1(s)\right)\left(f_2(t)-f_2(s)\right)&0\\
 0&\left(g_1(t)-g_1(s)\right)\left(g_2(t)-g_2(s)\right)
 \end{array}\right)
\end{align*}}
is positive. Using Theorem \ref{20} we have the inequality
{\footnotesize\begin{align*}
&\left(\begin{array}{cc}
 \int_\mathbb{R} \alpha(s) ds\int_\mathbb{R}\beta(t)f_1(t)f_2(t) dt&0\\
 0&\int_\mathbb{R} \alpha(s) ds\int_\mathbb{R}\beta(t)g_1(t)g_2(t) dt
 \end{array}\right) \\&\quad+ \left(\begin{array}{cc}
 \int_\mathbb{R} \beta(s) ds\int_\mathbb{R}\alpha(t)f_1(t)f_2(t) dt&0\\
 0&\int_\mathbb{R} \beta(s) ds\int_\mathbb{R}\alpha(t)g_1(t)g_2(t) dt
 \end{array}\right) \nonumber\\&\geq\left(\begin{array}{cc}
 \int_\mathbb{R}\int_\mathbb{R}\alpha(t)\beta(s)f_1(t)f_2(s) dtds&0\\
 0&\int_\mathbb{R}\alpha(t)\beta(s)g_1(t)g_2(s) dtds
 \end{array}\right)\\&\quad+\left(\begin{array}{cc}
 \int_\mathbb{R}\int_\mathbb{R}\alpha(s)\beta(t)f_1(t)f_2(s) dtds&0\\
 0&\int_\mathbb{R}\alpha(s)\beta(t)g_1(t)g_2(s) dtds
 \end{array}\right),
 \end{align*}}
 where $\alpha, \beta\in \mathbf{L}^1(\mathbb{R})$.
\end{example}

Let us consider $A_j,B_j\in{\mathbb B}({\mathscr H})\,\,(1\leq j\leq n)$  and let two nonnegative numbers $\omega_1, \cdots, \omega_n$, $\nu_1, \cdots, \nu_n$ such that $W_n=\sum_{ j=1}^n \omega_j$, $V_n=\sum_{ j=1}^n \nu_j$. We define the mapping $Q:{N}_+\times{N}_+\rightarrow {\mathbb B}({\mathscr H})$,
{\footnotesize\begin{align*}
 Q(k,n,A_j,B_j)&=W_k\sum_{ j=1}^k\nu_j (A_j\circ B_j)+V_k\sum_{ j=1}^n\omega_j (A_j\circ B_j)\\&\qquad +\Big{(}\sum_{ i=k+1}^n \omega_i A_i\Big{)}\circ\Big{(}\sum_{ j=1}^n
 \nu_j B_j\Big{)}+\Big{(}\sum_{ i=1}^n \nu_i A_i\Big{)}\circ\Big{(}\sum_{ j=k+1}^n
 \omega_j B_j\Big{)} \\&\qquad +\Big{(}\sum_{ i=k+1}^n \nu_i A_i\Big{)}\circ\Big{(}\sum_{ j=1}^n
 \omega_j B_j\Big{)}+\Big{(}\sum_{ i=1}^n\omega_i A_i\Big{)}\circ\Big{(}\sum_{ j=k+1}^n
 \omega_j B_j\Big{)},
\end{align*}}
where $k=1,2,\cdots, n,$ and
{\footnotesize\begin{align}\label{pool}
\sum_{ j=n+1}^n\omega_jA_j=\sum_{ j=n+1}^n\omega_jB_j=\sum_{ j=n+1}^n\nu_jA_j=\sum_{ j=n+1}^n\nu_j B_j=0
\end{align}}
 is assumed. Using the definition of $Q$ and the relation \eqref{pool} we get
 \\
$(a)\,\, Q(1,n,A_j,B_j)=\Big{(}\sum_{ i=1}^n \omega_i A_i\Big{)}\circ\Big{(}\sum_{ j=1}^n
 \nu_j B_j\Big{)}+\Big{(}\sum_{ i=1}^n \nu_iA_i\Big{)}\circ\Big{(}\sum_{ j=1}^n
 \omega_j B_j\Big{)}$\\
$(b)\,\,Q(n,n,A_j, B_j)=\sum_{ i=1}^n \omega_i\sum_{ j=1}^n\nu_j (A_j\circ B_j)+\sum_{ i=1}^n \nu_i\sum_{ j=1}^n\omega_j (A_j\circ B_j)$.\\
  \\
Now, in the next theorem we show a refinement of inequality \eqref{13}.
\begin{theorem}\label{31}
 Suppose that  $A_j,B_j\in{\mathbb B}({\mathscr H})\,\,(1\leq j\leq n)$ are selfadjoint operators  with the synchronous Hadamard property, $Q$ is defined as above,
 $\omega_1, \cdots, \omega_n$,  $\nu_1, \cdots, \nu_n$ are positive numbers. Then
{ \begin{align*}
 Q(n,n,A_j, B_j)\geq \cdots\geq Q(k,n,A_j,B_j)\geq\cdots\geq Q(1,n,A_j,B_j)
\end{align*}}
for each $k=1,2, \cdots, n$.
\end{theorem}
\begin{proof}
For all $k=2,3,\cdots, n$, we have
{\scriptsize\begin{align*}
 &Q(k,n,A_j,B_j)-Q(k-1,n,A_j,B_j)\\&=(W_{k-1}+\omega_k)\left(\sum_{ j=1}^{k-1}\nu_j (A_j\circ B_j)+\nu_kA_k\circ B_k\right)+(V_{k-1}+\nu_k)\left(\sum_{ j=1}^{k-1}\omega_j (A_j\circ B_j)+\omega_kA_k\circ B_k\right)\\&\qquad-
 \left[W_{k-1} \sum_{ j=1}^{k-1}\nu_j(A_j\circ B_j)+V_{k-1}\sum_{ j=1}^{k-1}\omega_j(A_j\circ B_j)\right]+\left(\sum_{ j=k+1}^{n}\omega_jA_j\right)\circ\left(\sum_{ j=1}^{n}\nu_jB_j\right)\\&\qquad+\left(\omega_kA_k+\sum_{ j=1}^{k-1}\omega_jA_j\right)\circ\sum_{ j=k+1}^{n}\nu_jB_j-\left(\omega_kA_k+\sum_{ j=k+1}^{n}\omega_jA_j\right)\circ\sum_{ j=1}^{n}\nu_jB_j\\&\qquad-\sum_{ j=1}^{k-1}\omega_jA_j\circ\left(\nu_kB_k+\sum_{ j=k+1}^{n}\nu_jB_j\right)+\left(\sum_{ j=k+1}^{n}\nu_j A_j\right)\circ\left(\sum_{ j=1}^{n}\omega_jB_j\right)\\&\qquad+
 \left(\nu_kA_k+\sum_{ j=1}^{k-1}\nu_jA_j\right)\circ\sum_{ j=k+1}^{n}\omega_jB_j-\left(\nu_kA_k+\sum_{ j=k+1}^{n}\nu_jA_j\right)\circ\sum_{ j=1}^{n}\omega_jB_j\\&\qquad-\sum_{ j=1}^{k-1}\omega_jB_j\circ\left(\omega_kB_k+\sum_{ j=1}^{n}\omega_jB_j\right)\\&=
 \left[ \omega_j\sum_{ j=1}^{k-1}\nu_j(A_j\circ B_j)+\omega_k(A_k\circ B_k)\sum_{ j=1}^{k-1}\nu_j-\left(\omega_kA_k\circ\sum_{ j=1}^{k-1}\nu_jB_j\right)-\left(\omega_kB_k\circ\sum_{ j=1}^{k-1}\nu_jA_j\right)\right]\\&\qquad+\left[\nu_k\sum_{ j=1}^{k-1}\omega_j(A_j\circ B_j)+\nu_k(A_k\circ B_k)\sum_{ j=1}^{k-1}\omega_j-\left(\nu_kA_k\circ\sum_{ j=1}^{k-1}\omega_jB_j\right)-\left(\nu_kB_k\circ\sum_{ j=1}^{k-1}\omega_jA_j\right)\right]\\&=\omega_k\sum_{ j=1}^{k-1}\nu_j(A_k-A_j)\circ(B_k-B_j)+\nu_k\sum_{ j=1}^{k-1}\omega_j(A_k-A_j)\circ(B_k-B_j)\geq0\\&(\normalsize\textrm{since the sequences}\hspace{.1cm}(A_j)_{j=1}^{n}\hspace{.1cm} \normalsize\textrm{and}\hspace{.1cm} (B_j)_{j=1}^{n}\hspace{.1cm}
 \normalsize\textrm{have synchronous Hadamard property}).
\end{align*}}
\end{proof}
\begin{remark}

If we put  $\omega_j=\nu_j\,\,(1\leq j\leq n)$ in Theorem \ref{31} then we get the inequalities
{ \begin{align*}
 q(n,n,A_j, B_j)\geq \cdots\geq q(k,n,A_j,B_j)\geq\cdots\geq q(1,n,A_j,B_j),
\end{align*}}
in which {\footnotesize\begin{align*}
 q(k,n,A_j,B_j)&=\sum_{ j=1}^k\omega_j\sum_{ j=1}^k\omega_j (A_j\circ B_j)+\Big{(}\sum_{ i=k+1}^n \omega_i A_i\Big{)}\circ\Big{(}\sum_{ j=1}^n
 \omega_j B_j\Big{)} \\&\qquad +\Big{(}\sum_{ i=k+1}^n \omega_i A_i\Big{)}\circ\Big{(}\sum_{ j=1}^n
 \omega_j B_j\Big{)}\qquad(k=1,2,\cdots, n).
\end{align*}}
\end{remark}
In next result, we show an extension of \eqref{bigbag} for interpolational means.
 \begin{theorem}\label{41}
Let ${\mathscr A}$ be a $C^*$-algebra, $T$ be a compact Hausdorff
space equipped with a Radon measure $\mu$ as a totaly order set, let
$(A_t)_{t\in T}, (B_t)_{t\in T} $ be positive increasing fields in
$\mathcal{C}(T,\mathscr A)$ and let $\alpha: T\rightarrow [0,
+\infty)$ be a measurable function. Then
{\scriptsize\begin{align*}
\int_{T} \alpha(s) d\mu(s)\int_{T}\alpha(t)(A_t\circ B_t) d\mu(t)\geq\Big{(}\int_{T}\alpha(t) (A_tm_{r,\alpha} B_t) d\mu(t)\Big{)}\circ\Big{(}\int_{T}\alpha(s) (A_sm_{r,1-\alpha} B_s) d\mu(s)\Big{)}
\end{align*}}
for all $\alpha\in[0,1]$ and all $r\in[-1,1]$.
\end{theorem}
\begin{proof}
 We have
 \begin{align}\label{path1}
 A_t\circ B_t&=(A_t\circ B_t)m_{r,\alpha}(A_t\circ B_t)=(U^*(A_t\otimes B_t)U)m_{r,\alpha}(U^*(B_t\otimes A_t)U)\nonumber\\&\geq U^*((A_t\otimes B_t)m_{r,\alpha}(B_t\otimes A_t))U\geq U^*((A_tm_{r,\alpha} B_t)\otimes(B_tm_{r,\alpha} A_t))U\nonumber\\&=(A_tm_{r,\alpha} B_t)\circ(B_tm_{r,\alpha} A_t)=(A_tm_{r,\alpha} B_t)\circ(A_tm_{r,1-\alpha} B_t),
 \end{align}
where $t\in T$; see \cite[p. 174]{abc}. Let $s, t\in T$. Without loss of generality assume that $s\preceq t$. Then by the property $(\textrm{i})$ of the operator mean, we have $0\leq(A_tm_{r,1-\alpha} B_t)-(A_sm_{r,1-\alpha} B_s)$ and $0\leq(A_tm_{r,\alpha} B_t)-(A_sm_{r,\alpha} B_s)$. Then
{\tiny\begin{align*}
 &\int_{T} \alpha(s) d\mu(s)\int_{T}\alpha(t)(A_t\circ B_t) d\mu(t)\\&\qquad-\Big{(}\int_{T}\alpha(t) (A_tm_{r,\alpha} B_t) d\mu(t)\Big{)}\circ\Big{(}\int_{T}\alpha(s) (A_sm_{r,1-\alpha} B_s) d\mu(s)\Big{)}\\&
=\int_{T} \int_{T}\alpha(s)\alpha(t)(A_t\circ B_t)
d\mu(t)d\mu(s)\\&\qquad-\int_{T}\int_{T}\alpha(t)\alpha(s)\Big{(}(A_tm_{r,\alpha}
B_t) \circ (A_sm_{r,1-\alpha} B_s)\Big{)} d\mu(t)d\mu(s)\,\, (\normalsize\textrm{by~} \ref{suit})\\&\geq
\int_{T} \int_{T}\alpha(s)\alpha(t)\Big{(}(A_tm_{r,\alpha} B_t)\circ
(A_tm_{r,1-\alpha}B_t)\Big{)}
d\mu(t)d\mu(s)\\&\quad-\int_{T}\int_{T}\alpha(t)\alpha(s)\Big{(}(A_tm_{r,\alpha}
B_t) \circ (A_sm_{r,1-\alpha} B_s)\Big{)} d\mu(t)d\mu(s)\\&\qquad\qquad\qquad\qquad\qquad\qquad\qquad\qquad
(\normalsize\textrm{by equation \eqref{path1}})\\&=\int_{T} \int_{T}\alpha(s)\alpha(t)\Big{(}(A_tm_{r,\alpha} B_t)\circ
(A_tm_{r,1-\alpha}B_t)\Big{)}-
\alpha(t)\alpha(s)\Big{(}(A_tm_{r,\alpha}
B_t) \circ (A_sm_{r,1-\alpha} B_s)\Big{)} d\mu(t)d\mu(s)\\&=\frac{1}{2}\Big[
\int_{T} \int_{T}\alpha(s)\alpha(t)\Big{(}(A_tm_{r,\alpha} B_t)\circ
(A_tm_{r,1-\alpha}B_t)\Big{)}
-\alpha(t)\alpha(s)\Big{(}(A_tm_{r,\alpha}
B_t) \circ (A_sm_{r,1-\alpha} B_s)\Big{)} d\mu(t)d\mu(s)\\&\quad+
\int_{T} \int_{T}\alpha(t)\alpha(s)\Big{(}(A_sm_{r,\alpha} B_s)\circ
(A_sm_{r,1-\alpha}B_s)\Big{)}
-\alpha(s)\alpha(t)\Big{(}(A_sm_{r,\alpha}
B_s) \circ (A_tm_{r,1-\alpha} B_t)\Big{)} d\mu(s)d\mu(t)\Big]\\&\qquad\qquad\qquad\qquad\qquad\qquad\qquad\qquad
(\normalsize\textrm{interchanging}\, s\, \normalsize\textrm{and}\, t\, \normalsize\textrm{in the second term})\\&={1\over2}\int_{T}
\int_{T}\Big{[}\alpha(s)\alpha(t)\Big{(}(A_tm_{r,t} B_t)\circ
(A_tm_{r,1-\alpha}B_t)\Big{)}
-\alpha(t)\alpha(s)\Big{(}(A_tm_{r,\alpha} B_t) \circ
(A_sm_{r,\alpha} B_s)\Big{)}\\&
\qquad+\alpha(t)\alpha(s)\Big{(}(A_sm_{r,\alpha} B_s)\circ
(A_sm_{r,1-\alpha}B_s)\Big{)}
-\alpha(s)\alpha(t)\Big{(}(A_sm_{r,\alpha}
B_s) \circ (A_tm_{r,1-\alpha} B_t)\Big{)} \Big{]} d\mu(t)d\mu(s)\\&\qquad\qquad\qquad\qquad\qquad\qquad\qquad\qquad
(\normalsize\textrm{by equation \eqref{suit}})\\&
={1\over2}\int_{T}\int_{T}\alpha(s)\alpha(t)\Big{[}(A_tm_{r,\alpha}
B_t)- (A_sm_{r,\alpha} B_s)\Big{]}\circ\Big{[}(A_tm_{r,1-\alpha}
B_t)- (A_sm_{r,1-\alpha} B_s)\Big{]} d\mu(t) d\mu(s)\\&\geq0
\qquad\qquad\qquad \qquad (\normalsize\textrm{by
the property (\textrm{i}) of the operator mean}).
\end{align*}}
\end{proof}
In the discrete case $T=\{1,\cdots, n\}$, if $\alpha(i)=\omega_i$ and $\beta(i)=\nu_i$, where $\omega_i, \nu_i\geq0$ $(1\leq i\leq n)$,
 then Theorem \ref{41} forces the next result.
 \begin{corollary}
Assume that  $A_j,B_j\in{\mathbb B}({\mathscr H})\,\,(1\leq j\leq n)$ are positive increasing operators and
 $\omega_1, \cdots, \omega_n$,  $\nu_1, \cdots, \nu_n$ are positive numbers. Then
{\footnotesize\begin{align*}
\sum_{i=1}^{n}w_i\sum_{i=1}^{n}\nu_i(A_i\circ B_i)\geq\Big{(}\sum_{i=1}^{n}w_i (A_im_{r,\alpha} B_i)\Big{)}\circ\Big{(}\sum_{j=1}^{n}\nu_j (A_jm_{r,1-\alpha} B_j)\Big{)}
\end{align*}}
for all $\alpha\in[0,1]$ and all $r\in[-1,1]$.
\end{corollary}

\bibliographystyle{amsplain}

\end{document}